\documentclass[a4paper,11pt]{amsart}

\usepackage{amsmath}
\usepackage{amsfonts}
\usepackage{amssymb}
\usepackage{graphicx}
\usepackage[abbrev]{amsrefs}

\usepackage{amsthm}
\usepackage[all,cmtip]{xy}
 
\newcommand{\Fp}{\mathbb{\overline{F}}_p}

\newcommand{\QQ}{\mathbb Q}

\newcommand{\OO}{\mathcal O}

\newcommand{\mbN}{\mathbb{N}}
\newcommand{\mbQ}{\mathbb{Q}}
\newcommand{\mbF}{\mathbb{F}}

\newcommand{\mbP}{\mathbb{P}}
\newcommand{\mcC}{\mathcal{C}}
\newcommand{\mcD}{\mathcal{D}}

\newcommand{\mcI}{\mathcal{I}}
\newcommand{\mcO}{\mathcal{O}}

\newcommand{\mfm}{\mathfrak{m}}

\DeclareMathOperator{\Supp}{Supp}

\DeclareMathOperator{\Aut}{Aut}

\DeclareMathOperator{\Pic}{Pic}
\DeclareMathOperator{\NE}{\overline{NE}}

\DeclareMathOperator{\Cone}{Cone}
\DeclareMathOperator{\coeff}{coeff}

\newcommand*{\coloneq}{\mathrel{\mathop:}=}

\theoremstyle{plain}
\newtheorem{theorem}{Theorem}[section]
\newtheorem{proposition}[theorem]{Proposition}
\newtheorem{lemma}[theorem]{Lemma}
\newtheorem{corollary}[theorem]{Corollary}

\newtheorem{question}[theorem]{Question}
\newtheorem{claim}[theorem]{Claim}

\theoremstyle{definition}

\newtheorem{example}[theorem]{Example}

\theoremstyle{remark}
\newtheorem{remark}[theorem]{Remark}

\title[On base point free theorem for lc threefolds over $\Fp$]
{On base point free theorem for log canonical threefolds over 
the algebraic closure of a finite field
}

\author{Diletta Martinelli}
\address{Department of Mathematics, Imperial College London, 180 Queen's
Gate, London, SW7 2AZ, UK.}
\email{d.martinelli12@imperial.ac.uk}

\author{Yusuke Nakamura}
\address{Graduate School of Mathematical Sciences, 
The University of Tokyo, 3-8-1 Komaba, Meguro-ku, Tokyo 153-8914, Japan.}
\email{nakamura@ms.u-tokyo.ac.jp}

\author{Jakub Witaszek}
\address{The London School of Geometry and Number Theory, Department of Mathematics, 
University College London, Gower Street, London, WC1E 6BT, UK.}
\email{jakub.witaszek.14@ucl.ac.uk}

\begin{document}
\begin{abstract}
We prove the base point free theorem for big line bundles on a three-dimensional log canonical 
projective pair defined over the algebraic closure of a finite field. 
This theorem is not valid for any other algebraically closed field. 
\end{abstract}


\maketitle


\section{Introduction}

A line bundle is called semiample if some positive tensor power 
$L^{\otimes r}$ is generated by global sections. 
Semiample line bundles play an important role in algebraic geometry, 
because they determine morphisms of a variety into projective spaces. 
Therefore, one would like to find necessary and sufficient conditions for semiampleness. 
A semiample line bundle is necessarily nef, but the converse is false in general. 
However, if we assume that $L$ is the canonical bundle and is nef, 
then the abundance conjecture \cite[Conjecture 3.12]{KollarMori} states that $L$ must be semiample.
Furthermore, the base point free theorem \cite[Theorem 3.3]{KollarMori} asserts that a nef line bundle $L$ 
on a Kawamata log terminal projective pair $(X, \Delta)$ 
defined over an algebraically closed field of characteristic zero 
is semiample, when $L - (K_X + \Delta)$ is nef and big. 

In positive characteristic, questions regarding semiampleness are more difficult, 
due to the absence of a proof of the resolution of singularities for varieties of dimension greater than 
three and the failure of the Kawamata--Viehweg vanishing theorem. 
As such, the base point free theorem remains still unsolved in general. 
However, many partial results for threefolds may be obtained by reductions to the two-dimensional cases.

The base point free theorem in positive characteristic 
is known for big line bundles $L$ when $(X, \Delta)$ is a three-dimensional Kawamata log terminal projective pair 
defined over an algebraically closed field of characteristic larger than five
(see \cite{Birkar} and \cite{Xu}). 
Over $\Fp$, the algebraic closure of a finite field, 
a stronger result is due to  Keel \cite{Keel} 
who proved the base point free theorem for big line bundles $L$ when $(X, \Delta)$ 
is a three-dimensional projective log pair defined over $\Fp$ 
with all the coefficients of $\Delta$ less than  one.

In this paper, we generalize Keel's result to the cases 
when the coefficients of $\Delta$ may be equal to one. 
Our main theorem is the following. 

\begin{theorem}\label{theorem:main}
Let $(X, \Delta)$ be a three-dimensional projective log pair defined over $\Fp$. 
Assume that one of the following conditions holds:
\begin{itemize}
\item[(1)] $(X, \Delta)$ is log canonical, or
\item[(2)] all the coefficients of $\Delta$ are at most one and each irreducible component of 
$\Supp (\lfloor \Delta \rfloor)$ is normal. 
\end{itemize}
Let $L$ be a nef and big line bundle on $X$. 
If $L - (K_X + \Delta)$ is also nef and big, then $L$ is semiample. 
\end{theorem}

The next corollary follows easily from Theorem \ref{theorem:main}.  
\begin{corollary}\label{corollary:abundance}
Let $(X, \Delta)$ be a three-dimensional log canonical projective pair defined over $\Fp$. 
Then the following hold:
\begin{enumerate}
\item[(1)] If $K_X + \Delta$ is nef and big, then $K_X + \Delta$ is semiample. 
\item[(2)] If $-(K_X + \Delta)$ is nef and big, then $-(K_X + \Delta)$ is semiample.
\end{enumerate}
\end{corollary}

\begin{remark}\label{remark:counterexample}
Theorem \ref{theorem:main} does not hold over fields $k \not = \Fp$ even in the two-dimensional case 
(Example \ref{example:tanaka}). 
Corollary \ref{corollary:abundance} (2) also does not hold over algebraically closed fields $k \not = \Fp$ 
(Example \ref{example:gongyo}). 

In Example \ref{example:counterex}, 
we give a counterexample to Theorem \ref{theorem:main} 
if one does not impose any conditions on the effective $\mbQ$-divisor $\Delta$. 
It is not clear whether the theorem remains true if we only assume that all the coefficients of $\Delta$ are at most one.
\end{remark}

We also prove the base point free theorem 
for normal surfaces defined over $\Fp$ without assuming bigness. 
\begin{theorem}
\label{theorem:genshokurov}
Let $X$ be a normal projective surface defined over $\Fp$ and let $\Delta$ be an effective $\mbQ$-divisor. 
Assume that we have a nef line bundle $L$ on $X$ such that $L -(K_X + \Delta)$ is also nef. 
Then $L$ is semiample.
\end{theorem}
\begin{remark}
Note that it is not true in general that nef line bundles on smooth surfaces over $\Fp$ are semiample 
(see Totaro's example in \cite{Totaro}).
\end{remark}
\begin{remark}\label{remark:qcartier}
Theorem \ref{theorem:main} and Theorem \ref{theorem:genshokurov} 
hold if we assume that $L$ is only a $\mbQ$-Cartier $\mbQ$-divisor. 
Note, that if $L$ and $L - (K_X+\Delta)$ are big and nef, then also \[
nL - (K_X + \Delta) = (n-1)L + \left(L - \left(K_X+\Delta\right)\right)
\]
is big and nef for any integer $n \geq 1$.  
\end{remark}

The paper is organized as follows:
in Section \ref{section:pre}, we review some definitions and facts from 
minimal model theory and about the conductor scheme. 
Further, we list some results from Keel \cite{Keel} and show lemmas necessary for the proof
of the main theorem.
In Section \ref{section:surface}, we prove the base point free theorem for surfaces under weaker assumptions 
(Theorem \ref{theorem:genshokurov}). 
In Section \ref{section:reduction}, generalizing the proof of \cite[Theorem 0.5]{Keel}, 
we reduce Theorem \ref{theorem:main} to showing that the line bundle $L |_{\Supp \lfloor \Delta \rfloor}$ is semiample 
(Theorem \ref{theorem:reduction}). 
If $\Supp \lfloor \Delta \rfloor$ is irreducible, 
we know that $L |_{\Supp \lfloor \Delta \rfloor}$ is semiample by Theorem \ref{theorem:genshokurov}. 
The non-irreducible case is treated in Section \ref{section:glue}. 
In order to generalize Theorem \ref{theorem:genshokurov} to the non-irreducible surfaces, 
we combine an idea from Fujino \cite{fujino1} and Tanaka \cite{TanakaAbundanceForSlc}, together with special
properties of varieties defined over $\Fp$, which are proved in Section \ref{section:pre}. 
In Section \ref{section:proof}, we complete the proof of Theorem \ref{theorem:main} and of Corollary 
\ref{corollary:abundance}. 
In Section \ref{section:example}, we give counterexamples as 
it was stated in Remark \ref{remark:counterexample}.

\subsection*{Notation and conventions}
\begin{itemize}
\item When we work over a normal variety $X$, we often identify a line bundle $L$ 
with the divisor corresponding to $L$. 
For example, we use the additive notation $L + A$ for a line bundle $L$ and a divisor $A$. 

\item Following the notation of \cite{Keel}, for a morphism $f \colon X \to Y$, a line bundle $L$ on $Y$, and a section $s \in H^0 (Y, L)$, 
we denote by $L |_X$ and $s |_X$ the pullbacks $f^* L$ and $f^* s$, respectively. 

\item With the same notation as above, 
we say that a section $t \in H^0(X, L|_X)$ \textit{descends to $Y$} if there exists a section 
$s \in H^0(Y,L)$ such that $f^* s = t$. 

\item Let $X$ be a reduced scheme of finite type over a field, 
$X = \bigcup X_i$ the decomposition into irreducible components, and 
$\overline{X_i} \to X_i$ the normalizations. 
Then we define the \textit{normalization of} $X$ as the composition 
$\bigsqcup \overline{X_i} \to \bigsqcup X_i \to X$. 

\item Let $X$ be a scheme and $F \subset X$ a closed subscheme. 
Let $L$ be a line bundle on $X$ and $s \in H^0(X, L)$ its section. 
We say that \textit{$s$ is nowhere vanishing on $F$} if $s | _{\{x\}}$ is not zero 
as an element in the one-dimensional vector space $H^0 (\{x\}, L|_{\{x\}})$ 
for any closed point $x \in F$. 

\item We say that a line bundle $L$ on $X$ is \textit{semiample} when the linear system $|mL|$ is base point free
for a sufficiently large and divisible positive integer $m$. 
When $L$ is semiample, the surjective map $f \colon X \to Y$, defined by $|mL|$, satisfies 
$f_* \mcO _X = \mcO _Y$ for a sufficiently large and divisible positive integer $m$. 
We call $f$ \textit{the map associated to $L$}. 
\end{itemize}

\section*{Acknowledgements}
We would like to thank Paolo Cascini and Yoshinori Gongyo for suggesting this problem and for
their constant support and guidance.

The project started during the Pragmatic Research Summer School 2013 in Catania and 
part of the work carried on during the visit of the first and third author to the University of Tokyo.
We are grateful to the University of Catania and organizers of Pragmatic Research School 
as well as to Yujiro Kawamata and Yoshinori Gongyo for the invitation to the University of Tokyo.
	 
We are grateful to Andrea Fanelli, Enrica Floris, Atsushi Ito, and Hiromu Tanaka, for useful comments and suggestions. 

The first author is supported by a Roth studentship.
The second author is supported by the Grant-in-Aid for Scientific Research
(KAKENHI No. 25-3003) 
and the Program for Leading Graduate Schools, MEXT, Japan. 
The third author was supported by Bonn International Graduate School's Pre-PhD scholarship.

\section{Preliminaries}\label{section:pre}
\subsection{Log pairs}\label{section:log_pair}
A \textit{log pair} $(X, \Delta)$ is a normal variety $X$ and an effective $\mbQ$-divisor $\Delta$ such that 
$K_X + \Delta$ is $\mbQ$-Cartier. 

For a proper birational morphism $f \colon X' \to X$ from a normal variety $X'$, 
we write 
\[
K_{X'} + \sum _i a_i E_i = f^* (K_X + \Delta), 
\]
where $E_i$ are prime divisors. 
We say that the pair $(X, \Delta)$ is \textit{log canonical} 
if $a_i \le 1$ 
for any proper birational morphism $f$. 
Further, 
we say that the pair $(X, \Delta)$ is \textit{Kawamata log terminal} 
if $a_i < 1$ 
for any proper birational morphism $f$.

\subsection{Conductor schemes}

Let $X$ be a reduced scheme of finite type over a field and $\overline{X} \to X$ its normalization. 
We identify the sheaf of rings $\mcO _X$ as the subring of $\mcO_{\overline{X}}$. 
Let $\mcI \subset \mcO _X$ be the maximal ideal sheaf satisfying
$\mcI \mcO_{\overline{X}} \subset \mcO_{X}$. 
The \textit{conductor} of $X$ is the subscheme $\mcD \subset X$ defined by $\mcI$. 
By abuse of notation, the subscheme $\mcC \subset \overline{X}$ defined by 
$\mcI \mcO_{\overline{X}}$ will also be called the conductor. 

The notion of conductor is important to descend sections, 
because of the following remark:
\begin{remark}\label{remark:conductor}
Let $\mcC \subset \overline{X}$, $\mcD \subset X$ be the conductors
and let $L$ be a line bundle on $X$. 
\[\xymatrix{
\mcC \ar[d]  \ar@<-0.5mm>@{^(->}[r] & \overline{X} \ar[d] \\
\mcD \ar@<-0.5mm>@{^(->}[r] & X \\
}\]

\noindent
By definition of the conductor, we have the following exact sequence
\[
0 \longrightarrow H^0 (X, L) \longrightarrow 
H^0 (\overline{X}, L | _{\overline{X}}) \oplus H^0 (\mcD , L|_{\mcD})
\longrightarrow H^0(\mcC, L|_{\mcC}), 
\]
where the second map is defined by $t \mapsto (t|_{\overline{X}}, t|_{\mcD})$, and 
the third map is defined by $(t,u) \mapsto t | _{\mcC} - u| _{\mcC}$. 
Therefore, a section $s \in H^0 (\overline{X}, L | _{\overline{X}})$
descends to $X$ if and only if $s|_{\mcC}$ descends to $\mcD$. 
\end{remark}

\subsection{Adjunction formula} \label{subsection:adjunction}
Let $(X, \Delta)$ be a log pair and $S$ the union of the supports 
of some of the divisors with coefficient one in $\Delta$. 
Let $p \colon \overline{S} \to S$ be the normalization of $S$. 
Then there exists an effective $\mbQ$-divisor $\Delta _{\overline{S}}$ on $\overline{S}$ such that 
\[
K_{\overline{S}} + \Delta _{\overline{S}} = (K_X + \Delta)|_{\overline{S}}
\]
holds (see for instance \cite[Definition 4.2]{Kollar}). 

We denote by $C$ the possibly non-reduced divisor on $\overline{S}$ 
corresponding to the codimension one part of $\mcC$, 
where $\mcC \subset \overline{S}$ is the conductor of $S$. 

When $X$ is $\mbQ$-factorial, it follows that $C \le \Delta _{\overline{S}}$ by \cite[Theorem 5.3]{Keel}. 
In this paper, we use the following proposition, which only states 
$\Supp (C) \subset  \Supp (\lfloor \Delta _{\overline{S}} \rfloor)$, but 
is valid even for a non-$\mbQ$-factorial variety $X$. 
\begin{proposition}\label{proposition:adjunction}
Let $(X, \Delta)$ be a log pair, and 
let $S$ be the union of the supports 
of some of the divisors with coefficient one in $\Delta$. 
Let $p \colon \overline{S} \to S$ be the normalization of $S$, and 
let $\Delta _{\overline{S}}$ be an effective $\mbQ$-divisor on $\overline{S}$
defined by the adjunction as above. 
Further, we denote by $C$ the (possibly non-reduced) divisor on $\overline{S}$ 
corresponding to the codimension one part of $\mcC$, 
where $\mcC \subset \overline{S}$ is the conductor of $S$. 
Then the following hold: 
\begin{enumerate}
\item $\Supp (C) \subset  \Supp (\lfloor \Delta _{\overline{S}} \rfloor)$. 
\item Let $D_1, \ldots, D_c$ be prime divisors with coefficient greater 
than or equal to one in $\Delta$, and 
let $T = \bigcup _{1 \le i \le c} \Supp (D_i)$. 
Assume that each $D_i$ satisfies $\Supp (D_i) \not \subset S$. 
Then, the codimension one part of $p^{-1}(S \cap T)$ 
is contained in $\Supp (\lfloor \Delta _{\overline{S}} \rfloor)$. 
\end{enumerate}
\end{proposition}
\begin{proof}
First, we prove (1). 
Let $V \subset \overline{S}$ be a codimension one subvariety such that $V \subset \mcC$. 
It is sufficient to show $\coeff _V \Delta _{\overline{S}} \ge 1$. 
When $(X, \Delta)$ is not log canonical at the generic point $\eta _{p(V)}$ of $p(V)$, 
we have $\coeff _V \Delta _{\overline{S}} > 1$ (see \cite[Proposition 4.5 (2)]{Kollar}). 
Hence, we may assume that $(X, \Delta)$ is log canonical at $\eta _{p(V)}$. 
In this case, $S$ has a node at $\eta _{p(V)}$ and $\coeff _V \Delta _{\overline{S}} = 1$ 
(see the proof of \cite[Proposition 4.5 (6)]{Kollar}).  

Next, we prove (2). 
Let $V \subset \overline{S}$ be a codimension one subvariety such that $V \subset p^{-1}(S \cap T)$. 
It is sufficient to show $\coeff _V \Delta _{\overline{S}} \ge 1$. 
Since the problem is local around $V$, we may assume that $p(V) \subset \Supp (D_i)$ for all $i$. 
If $\coeff _{D_i} \Delta > 1$ for some $i$, then $(X, \Delta)$ is not log canonical at the generic point $\eta _{p(V)}$ of $p(V)$. 
In this case, we have  $\coeff _V \Delta _{\overline{S}} > 1$ as above. 
Hence, we may assume that $\coeff _{D_i} \Delta = 1$ for all $i$. 
Note that $S \cap T$ is contained in the conductor of the normalization of $S \cup T$. 
Therefore, we conclude the proof by applying (1) to $S \cup T$. 
\end{proof}

\subsection{Some properties of varieties over $\Fp$}
The following fact distinguishes $\Fp$ from other fields of positive characteristic. 
For the proof, see for instance \cite[Lemma 2.16]{Keel}. 

\begin{proposition}\label{fact::torsionpic0}
The Picard scheme $\Pic ^0 X$ is a torsion group when $X$ is a projective scheme defined over $\Fp$. 
In particular, any numerically trivial Cartier divisor is $\mbQ$-linearly trivial. 
\end{proposition}

We need the following lemmas in Section \ref{section:glue}. 

\begin{lemma}\label{lemma:0dim}
Let $X$ be a proper scheme over $\overline{\mbF} _p$. 
Let $s_1, s_2 \in H^0(X, \mcO _X)$ be sections of the structure sheaf. 
Assume that $s_1$ and $s_2$ are nowhere vanishing on $X$.
Then there exists $n \ge 1$ such that $s_1 ^n = s_2 ^n$ in $H^0(X, \mcO _X)$. 
\end{lemma}
\begin{proof}
Without loss of generality we may assume that $X$ is connected. Set $A := H^0(X, \mcO _X)$. 
It is a finite-dimensional vector space over $\overline{\mbF} _p$, because $X$ is proper. 
Since $X$ is connected, $A$ has the unique maximal ideal $\mfm$, 
and it follows that 
$A/\mfm \cong  H^0(X^{\mathrm{red}}, \mcO _{X^{\mathrm{red}}}) \cong \overline{\mbF} _p$.

Let $a_i$ be the element of $A$ corresponding to $s_i$ and $\overline{a_i}$ the image of $a_i$ 
in $\overline{\mbF} _p$. 
Since $s_i$ is nowhere vanishing on $X$, the element $\overline{a_i} \in \overline{\mbF} _p$ is not zero. 
Hence, there exists $e \ge 1$ for which $\overline{a_1} ^{p^e -1} = \overline{a_2} ^{p^e -1} =1$. 
Take $r \ge 1$ such that $\mfm ^{p^r} = 0$. 
Then we have 
\[
a_1 ^{p^r(p^e -1)} - a_2^{p^r (p^e -1)} = \big( a_1 ^{p^e -1} - a_2 ^{p^e -1} \big)^{p^r} 
\in \mfm^{p^r} = 0.
\]
Therefore, it is sufficient to set $n = p^r(p^e -1)$. 
\end{proof}

\begin{lemma}\label{lemma:descent}
Let $X$ be a one-dimensional reduced scheme of finite type over $\Fp$, 
$L$ a line bundle on $X$, and $p \colon \overline{X} \rightarrow X$ the normalization of $X$. 
Let $\mcC \subset \overline{X}$ be the conductor of $X$, and 
$s \in H^0(\overline{X}, L|_{\overline{X}})$ be a section nowhere vanishing on $\mcC$. 
Then $s^n$ descends to $X$ for some $n \ge 1$.
\end{lemma}

\begin{proof}
Let $\mcD \subset X$ be the conductor. 
Note that $\mcC$ and $\mcD$ are either empty or have dimension zero. 
By Remark \ref{remark:conductor}, it is sufficient to prove that $s^n|_{\mcC}$ descends to $\mcD$ for some $n \ge 1$. 
Let $t \in H^0 (\mcD , L|_{\mcD})$ be a section nowhere vanishing on $\mcD$. 
Then $t|_ \mcC$ is nowhere vanishing on $\mcC$. 
Any line bundle is trivial on a zero-dimensional scheme, and so by Lemma \ref{lemma:0dim}, 
we get $s^n |_{\mcC} = t^n |_{\mcC}$ for some $n \ge 1$. 
In particular, $s^n|_{\mcC}$ descends to $\mcD$.
\end{proof}

\begin{lemma}\label{lemma:fin_group}
Let $C$ be a smooth proper connected curve over $\Fp$. 
Then, a finitely generated subgroup of $\Aut (C)$ is finite. 
\end{lemma}

\begin{proof}
If $g(C) \ge 2$, then $\Aut (C)$ is finite and the statement is trivial. 
If $C = \mbP ^1$, then $\Aut (C) \cong \operatorname{PGL}(2, \Fp)$. 
A finitely generated subgroup $G$ of $\operatorname{PGL}(2, \Fp)$ is always finite, 
because $G$ is contained in $\operatorname{PGL}(2, \mbF _{p^e})$ for some $e \ge 1$. 
If $C$ is an elliptic curve, then we get $\Aut (C) \cong T \rtimes F$, 
where $T$ is the group of translations and $F$ is a finite group 
(see for instance \cite[Section X.5]{Silverman}). 
Note that each element of $T$ has finite order, because $C$ is defined over $\Fp$. 
Hence, a finitely generated subgroup of the abelian group $T$ is always finite, 
and so a finitely generated subgroup of $\Aut (C)$ is also finite. 

For completeness, we note a general fact in group theory: 
any finitely generated subgroup of $G_1 \rtimes G_2$ is finite, 
if we assume that for each $i$, any finitely generated subgroup of $G_i$ is finite. 
\end{proof}

\subsection{Keel's theorem}
In this subsection, we list some theorems from Keel \cite{Keel}. 

The following theorem is crucial in reducing problems from threefolds to surfaces. 
\begin{theorem}[Keel {\cite[Proposition 1.6]{Keel}}]
\label{theorem:basereduction}
Let X be a projective scheme over a field of positive characteristic. 
Let $L$ be a nef line bundle on $X$, and let $E$ be an effective Cartier divisor on $X$
such that $L-E$ is ample. 
Then $L$ is semiample if and only if $L|_{E_{\text{red}}}$ is semiample.
\end{theorem}
We note that Cascini, M\textsuperscript{c}Kernan, and Musta{\c{t}}{\u{a}} 
gave a different proof of Theorem \ref{theorem:basereduction} \cite[Theorem 3.2]{CMM}.

\begin{theorem}[Artin {\cite[Theorem 2.9]{Artin}}, Keel {\cite[Corollary 0.3]{Keel}}]\label{theorem:big}
Let $X$ be a projective surface over $\Fp$ and let $L$ be a nef and big line bundle on $X$. 
Then $L$ is semiample.
\end{theorem}
\begin{proof}
Since by Proposition \ref{fact::torsionpic0} nef line bundles on curves over $\Fp$ are semiample, 
the claim follows from Theorem \ref{theorem:basereduction}. 
\end{proof}

We say that a map $f \colon X \to Y$ is a \textit{finite universal homeomorphism} 
if it is a finite homeomorphism under any base change. 
In this case, we have a correspondence, up to taking powers, between 
the set of sections of a line bundle $L$ on $Y$ and the set of sections of $L|_X$. 
\begin{theorem}[Keel {\cite[Lemma 1.4]{Keel}}]\label{theorem:univhomeo}
Let $f \colon X \to Y$ be a finite universal homeomorphism 
between varieties defined over a field of characteristic $p>0$ and let $L$ be a line bundle on $Y$. 
Then the following hold. 
\begin{enumerate}
\item[(1)] For $s \in H^0(X, L|_X)$, the section $s^{p^e} \in H^0(X, L^{\otimes p^e}|_X)$ 
descends to $Y$ for a sufficiently large integer $e \ge 1$. 
\item[(2)] If $t \in H^0(Y, L)$ satisfies $t|_X = 0$, then $t^{p^e} = 0$ holds for a sufficiently large integer $e \ge 1$. 
\end{enumerate}
\end{theorem}

In this paper, 
we frequently use the following theorems. 

\begin{theorem}[Keel {\cite[Corollary 2.12]{Keel}}]\label{theorem:gluing1}
Let $X = X_1 \cup X_2$ be a projective scheme over $\Fp$, where $X_i$ are closed subsets. 
Let $L$ be a nef line bundle on $X$ such that $L|_{X_i}$ are semiample. 
Let $g_i \colon X_i \to Z_i$ be the map associated to $L|_{X_i}$.  
Assume that all but finitely many fibers of $g_2|_{X_1 \cap X_2}$ are geometrically connected. 
Then $L$ is semiample.
\end{theorem}

\begin{theorem}[Keel {\cite[Corollary 2.14]{Keel}}]
\label{theorem:gluing2}
Let $X$ be a reduced projective scheme over $\Fp$. 
Let $p \colon \overline{X} \rightarrow X$ be the normalization of $X$. 
Let $D \subset X$ and $C \subset \overline{X}$ be the reductions of the conductors. 
Let $L$ be a nef line bundle on $X$ such that $L|_{\overline{X}}$ and $L|_{D}$ are semiample. 
Let $g \colon \overline{X} \to Z$ be the map associated to $L|_{\overline{X}}$.  
Assume that all but finitely many fibers of $g|_{C}$ are geometrically connected. Then $L$ is semiample.
\end{theorem}

\section{Base point free theorem for normal surfaces}\label{section:surface}
In this section, we prove Theorem \ref{theorem:genshokurov}. 
The key tool is the following theorem of Tanaka. 
We say that a $\mbQ$-divisor $B$ on a variety $X$ is \textit{$\mbQ$-effective} if 
$h^0(X,mB) \neq 0$ for some $m \geq 1$.
Note that a normal surface over $\Fp$ is always $\mbQ$-factorial (see \cite[Theorem 11.1]{Tanaka}). 

\begin{theorem}[Tanaka {\cite[Theorem 12.6]{Tanaka}}]\label{theorem:tanaka_bpf}
Let $X$ be a projective normal surface over $\Fp$ and let D be a nef divisor. 
If $qD - K_X$ is $\mbQ$-effective 
for some positive rational number $q \in \mathbb{Q}$, then D is semiample.
\end{theorem}

We will use the following proposition
to reduce the case of hyperelliptic surfaces to abelian surfaces. 
 
\begin{proposition}
\label{proposition:cover}
Let $p \colon Y \rightarrow X$ be a proper surjection between varieties defined over an algebraically closed field and 
let $L$ be a line bundle on $X$. 
Assume that $X$ is normal. Then $L$ is semiample if and only if $p^*(L)$ is semiample.
\end{proposition}
\begin{proof}
See for instance the proof of \cite[Lemma 2.10]{Keel}.
\end{proof}

\begin{proof}[Proof of Theorem \ref{theorem:genshokurov}]
Recall that we have the nef line bundle $L$ and the $\mbQ$-divisor $D \coloneq L - (K_X + \Delta)$ on 
the normal surface $X$ over $\Fp$.
\begin{claim} 
We can assume that $X$ is smooth. 
\end{claim}
\begin{proof}
	Let $f \colon Y \rightarrow X$ be the minimal resolution of singularities of $X$. 
	Define $\Delta_Y$ so that $K_Y + \Delta_Y = f^*(K_X + \Delta)$. 
	The divisor $\Delta_Y$ is an effective $\mbQ$-divisor by the negativity lemma 
	(cf.~ \cite[Corollary 4.3]{KollarMori}).
	Note that $f^*L$ and $f^*D = f^*L - (K_Y+\Delta_Y)$ are nef. 
	By Proposition \ref{proposition:cover} we know that $L$ is semiample if and only if $f^*L$ is semiample. 
	Thus, by replacing $X$ by $Y$, we may assume that the surface is smooth.  
\end{proof}

We extensively use the following lemma. 
\begin{lemma} \label{lemma:d_effective}
If $D$ is $\mbQ$-effective, then $L$ is semiample.
\end{lemma}
\begin{proof}
	As $D$ is $\QQ$-effective, then also $L-K_X = D + \Delta$ is $\QQ$-effective, 
	and so $L$ is semiample by Theorem \ref{theorem:tanaka_bpf}. 
\end{proof} 

\begin{claim} \label{claim:genshokurovmain} 
We can assume that all the following statements are true. 

\noindent
\begin{tabular}{lll}
\normalfont{(1)} $L \not \equiv 0$ and $D \not \equiv 0$, &
\normalfont{(2)} $L^2=0$, &
\normalfont{(3)} $D^2=0$, \\
\normalfont{(4)} $L \cdot \Delta = 0$, &
\normalfont{(5)} $L \cdot K_X = 0$, \qquad &
\normalfont{(6)} $(K_X + \Delta)\cdot \Delta = 0$, \\
\normalfont{(7)} $(K_X + \Delta)\cdot K_X = 0$, \qquad &
\normalfont{(8)} $\chi(\OO_X) \leq 0$. &
\end{tabular}
\end{claim}
\begin{proof}
	
	If $L \equiv 0$, then $L \sim_{\QQ} \OO_X$ by Proposition \ref{fact::torsionpic0}, 
	so $L$ is semiample. Thus, we may assume that $L \not \equiv 0$. 
	Analogously, we may assume that $D \not \equiv 0$.
	
	As $L$ and $D$ are nef, we get $L^2 \geq 0$ and $D^2 \geq 0$. If $L^2 > 0$, then, 
	by Theorem \ref{theorem:big}, the line bundle $L$ is semiample. Thus, we may assume that $L^2=0$. 
	If $D^2>0$, then $D$ is big, and so $\mathbb{Q}$-effective. 
	In this case $L$ is semiample by Lemma \ref{lemma:d_effective}. 
	Hence, we may assume $D^2 = 0$.
	
	Since $L \not \equiv 0$, we know that there exists a curve $C$ on $X$ satisfying $L \cdot C >0$. 
	Take an ample divisor $A$ such that $A-C$ is effective. Then $L \cdot A = L \cdot C + L \cdot (A-C) > 0$. 
	If $m$ is sufficiently large so that it satisfies $(K_X - mL) \cdot A < 0$, 
	then $h^2(X,mL) = h^0 (X,K_X - mL) = 0$. 
	The Riemann-Roch theorem gives
	\begin{align*}
	h^0(X,mL) &= h^1(X,mL) + \frac{1}{2} mL \cdot (mL - K_X) + \chi(\OO_X)  \\
	 		  &= h^1(X,mL) - \frac{1}{2} mL \cdot K_X + \chi(\OO_X). 
	\end{align*}
	
	\noindent
	As $L$ and $D$ are nef, it follows that
	\[
	0 \leq L \cdot D = -L \cdot K_X - L \cdot \Delta.
	\]
Since $\Delta$ is effective and $L$ is nef, we find $0 \leq L \cdot D \leq -L \cdot K_X$. 
If $-L\cdot K_X > 0$, then $\kappa(X,L) = 1$ by the calculation of $h^0(X, mL)$ above. 
A nef line bundle $L$ with $\kappa(X,L) = 1$ is always semiample (see for instance \cite[Theorem 11.3.1]{FaA}). 
Thus, we may assume that $L \cdot \Delta = 0$ and $L \cdot K_X =0$.	

As above, $h^2(X,mD) = 0$ holds 
for sufficiently large $m$, and so the Riemann-Roch theorem gives
\begin{align*}
h^0(X,mD) &= h^1(X,mD) - \frac{1}{2} mD \cdot K_X + \chi(\OO_X) \\
		  &= h^1(X,mD) + \frac{1}{2} mD \cdot (D - L + \Delta) + \chi(\OO_X) \\
		  &= h^1(X,mD) + \frac{1}{2} mD \cdot \Delta + \chi(\OO_X) \\
		  &= h^1(X,mD) - \frac{1}{2} m(K_X + \Delta) \cdot \Delta + \chi(\OO_X).
\end{align*}
If $-(K_X + \Delta) \cdot \Delta > 0$, then $D$ is $\QQ$-effective and 
by Lemma \ref{lemma:d_effective} the line bundle $L$ is semiample. 
Since $0 \le D \cdot \Delta = -(K_X + \Delta) \cdot \Delta$ holds by the nefness of $D$, 
we may assume $(K_X + \Delta) \cdot \Delta = 0$. 
Given $D^2 = L^2 = D \cdot L = 0$, it follows that $(K_X + \Delta) \cdot K_X = 0$. 

By the Riemann-Roch theorem, we get $h^0 (X, mD) = h^1(X, mD) + \chi(\mcO _X)$. 
If $\chi(\OO_X) > 0$, then $D$ is $\QQ$-effective 
and by Lemma \ref{lemma:d_effective} the line bundle $L$ is semiample. 
Hence, we may assume that $\chi(\OO_X) \leq 0$.  
\end{proof}

We divide the proof into cases depending on the Kodaira dimension.

\vspace{2mm}

\noindent
\textbf{\underline{Case 1.}}\ \ 
Assume $\kappa(X) \geq 0$.

\begin{claim} \label{claim:genshokurovnonminimal} 
We may assume that $K_X$ is nef. 
\end{claim}
\begin{proof}
Let $\pi \colon X \rightarrow X_{\text{min}}$ be the minimal model of $X$. 
By $\pi_* L$ we denote the pushforward of $L$ as a divisor. 

By the assumption $\kappa (X) \ge 0$, we have that 
$K_X$ is $\mbQ$-linearly equivalent to an effective 
$\mbQ$-divisor containing every $\pi$-exceptional curve in its support. 
Since $L \cdot K_X = 0$ and $L$ is nef, 
it follows that $L \cdot E = 0$ for every $\pi$-exceptional curve $E$. 
Hence, we get $L = \pi^*\pi_*L$, by the negativity of the intersection form on the exceptional locus 
(cf.~ \cite[Lemma 3.40]{KollarMori}). 

Since $L = \pi^*\pi_*L$, it is sufficient to show the semiampleness of $\pi_* L$. 
Note that $\pi _* L$ and $\pi _* D$ are  nef, because $L$ and $D$ are nef. 
Further, we have $\pi _* D = \pi _* L - (K_{X_{\text{min}}} + \pi _* \Delta)$. 
Therefore, we can reduce the problem to the case of the minimal model $X_{\text{min}}$. 
\end{proof}

In what follows, we assume that $X$ is minimal. 
We use the classification of minimal surfaces in positive characteristic 
(see for instance \cite{Mumford}, \cite{BM2}, \cite{BM3}, and \cite{Liedtke}). 

\vspace{2mm}

\noindent
\textbf{\underline{Case 1.1.}}\ \ 
Assume $\kappa(X) = 2$. 

We can write $K_X = A + E$ 
for an ample $\mbQ$-divisor $A$ and an effective $\mbQ$-divisor $E$, because $K_X$ is big. 
Since $L$, $D$ are nef and $L \cdot K_X = D \cdot K_X = 0$, 
it follows that $L \cdot A = D \cdot A = 0$. Thus, $(L-D) \cdot A = (K_X + \Delta) \cdot A = 0$. 
We get a contradiction 
\[
0 < A^2 \leq (K_X + \Delta) \cdot A = 0. 
\]
Hence, there are no line bundles 
$L$ satisfying the assumptions in Claim \ref{claim:genshokurovmain}. 

\vspace{2mm}

\noindent
\textbf{\underline{Case 1.2.}}\ \ 
Assume $\kappa(X) = 1$. 

In our case, $K_X$ is semiample and it gives an elliptic or quasi-elliptic fibration $f \colon X \to B$. 
Let $F$ be its general fiber. 
Then $K_X \equiv a F$ holds
for some positive rational number $a$. 

Since $D \cdot K_X = 0$, it follows that $D \cdot F = 0$. 
Therefore, $D$ is $f$-numerically trivial by the nefness of $D$. 
Since $D$ is nef and $f$-numerically trivial, it satisfies $D \equiv b F$ for some $b \ge 0$, 
by Lemma \ref{lemma:Lehmann}. 
Hence, $D$ is $\mbQ$-effective by Proposition \ref{fact::torsionpic0}. 
Therefore, $L$ is semiample by Lemma \ref{lemma:d_effective}. 

\begin{lemma} \label{lemma:Lehmann}
Let $f \colon X \rightarrow B$ be a surjective morphism satisfying $f_* (\mcO _X) = \mcO _B$ 
from a smooth projective surface $X$ to a smooth projective curve $B$. 
Suppose that $L$ is an $f$-numerically trivial nef $\mbQ$-Cartier $\mbQ$-divisor. 
Then $L \equiv bF$ for some $b \geq 0$, where $F$ denotes a general fiber of $f$.
\end{lemma}
\begin{proof}
See for instance \cite[Lemma 2.4]{Lehmann}.
\end{proof}

\vspace{2mm}

\noindent
\textbf{\underline{Case 1.3.}}\ \ 
Assume $\kappa(X) = 0$.

By the classification of minimal surfaces, 
there are five possibilities: 
a K3 surface, an Enriques surface, an abelian surface, a hyperelliptic surface, or a quasi-hyperelliptic surface. 

If $X$ is a K3 surface or an Enriques surface, then $\chi(\OO_X)=2$ or $\chi(\OO_X)=1$, respectively,
which contradicts Claim \ref{claim:genshokurovmain}.

If $X$ is an abelian surface, then every nef divisor is numerically equivalent to a semiample divisor
(see Proposition \ref{proposition:abelian}). 
Therefore, $L$ is semiample by Proposition \ref{fact::torsionpic0}. 

If $X$ is a hyperelliptic surface, then $X$ is a finite quotient of an abelian surface by a finite group. 
Therefore, we have a surjective morphism $A \to X$ from an abelian surface $A$. 
Since $L|_{A}$ is a nef line bundle on an abelian surface, it is semiample (see Proposition \ref{proposition:abelian}).
Hence, $L$ is also semiample by Proposition \ref{proposition:cover}.

If $X$ is a quasi-hyperelliptic surface, then $X$ can be written as a finite quotient $E \times C \to X$, 
where $E$ is an elliptic curve and $C$ is a rational curve with a cusp. 
Therefore, we have a surjective morphism $X' \coloneq E \times \mbP ^1 \to X$. 
Any divisor on $X'$ is numerically equivalent to $a F_1 + b F_2$ with $a, b \in \mbQ$, where 
$F_1$ is the fiber class of $X' \to E$ and $F_2$ is the fiber class of $X' \to \mbP ^1$. 
Hence, any nef divisor on $X'$ is numerically equivalent to a semiample divisor. 
Thus, we can conclude that $L$ is semiample by Proposition \ref{fact::torsionpic0} and Proposition \ref{proposition:cover}.

\vspace{2mm}

\noindent
\textbf{\underline{Case 2.}}\ \ 
Assume $\kappa(X) = -\infty$. 

By $\chi(\mcO _X) \le 0$, the surface $X$ is irrational. 
Thus, we can assume that $f \colon X \rightarrow B$ is a birationally ruled surface, 
where $B$ is a curve with $g(B) \geq 1$.

We need the following lemma, which can be found in the proof of \cite[Theorem 12.4]{Tanaka}. 

\begin{lemma} 
Let $C$ be an $f$-horizontal curve on $X$ such that $D \cdot C =0$. 
Then $D$ is $\QQ$-effective. 
\end{lemma}
\begin{proof}  
Since $C$ is a horizontal curve, it holds that $g(B) \le h^1(C, \mcO _C)$. 
By the Riemann-Roch theorem, we get 
\[
h^0(X, mD) = h^1(X, mD) + \chi (\mcO _X) = h^1(X, mD) + 1 - g(B), 
\]
so it is sufficient to show $h^1(X, mD) \ge h^1(C, \mcO _C)$ for some $m > 0$. 

Since $D \cdot C = 0$, we have $D |_{C} \equiv 0$. 
Hence, by Proposition \ref{fact::torsionpic0} we can conclude that 
$mD |_{C}$ is trivial for sufficiently divisible $m>0$. 
Therefore, we get an exact sequence
\[
0 \longrightarrow \mcO _X (mD-C) \longrightarrow \mcO _X (mD) \longrightarrow \mcO _C \longrightarrow 0. 
\]

By the same reason as before, $h^2(X, mD -C) = 0$ holds for sufficiently large $m$. 
Hence, we get $h^1(X, mD) \ge h^1(C, \mcO _C)$. 
\end{proof}
For any irreducible component $C$ of $\Delta$, it follows that $D \cdot C = 0$, because $D$ is nef and $D \cdot \Delta = 0$. 
In particular, if $\Delta$ has an $f$-horizontal component, 
then the lemma above implies that $D$ is $\QQ$-effective, 
and hence $L$ is semiample by Lemma \ref{lemma:d_effective}. 
Thus, in what follows, we may assume that $\Delta$ has only $f$-vertical components. 

\begin{claim} 
Under these assumptions, it follows that $\Delta=0$, $g(B)=1$, 
and $X$ is a minimal ruled surface. 
\end{claim}
\begin{proof} 
Let $\pi \colon X \rightarrow X_{\text{min}}$ be a minimal model of $X$.
We have $K_X \sim \pi^*K_{X_{\text{min}}} + E$, 
where $E$ is an exceptional divisor. 
We refer the reader to \cite[Ch V, Section 2]{Hartshorne} 
for the properties of ruled surfaces. 
It holds that 
\[
K_{X_{\text{min}}} \equiv -2C_0 + (2g(B)-2 - e)F
\]
for $C_0$ a normalized section, 
$e=-C_0^2$, and $F$ a general fiber of $X_{\text{min}} \to B$. 
Note that $K_{X_{\text{min}}}^2 = 8(1-g(B))$.

Since $(K_X+\Delta)\cdot \Delta=0$ and $(K_X+\Delta) \cdot K_X = 0$, we get 
\[
\Delta^2  = -K_X \cdot \Delta = K_X^2.
\]
As $\Delta$ has only $f$-vertical components, we have $\pi^*F \cdot \Delta = 0$, and so
\[
0 = (K_X + \Delta) \cdot \Delta = -2\pi^*C_0 \cdot \Delta + (E + \Delta) \cdot \Delta.
\]
Since $\pi^*C_0 \cdot \Delta \geq 0$, it follows that $E \cdot \Delta \geq -\Delta^2 $. 
Therefore,
\begin{align*}
(E + \Delta)^2 
&= E^2 + 2E \cdot \Delta + \Delta^2 \geq E^2 - \Delta^2 \\
&= E^2 - K_X^2 = - K_{X_{\text{min}}}^2 = 8(g(B)-1) \geq 0.
\end{align*}

By the Zariski lemma (cf.~ Theorem 1.23 in \cite[Section 9]{Liu}), 
the intersection form on $f$-vertical fibers is seminegative definite with one-dimensional 
radical equal to the span of a general fiber, so $(E + \Delta)^2=0$ and $E + \Delta \equiv \pi^*pF$ for some $p \in \QQ$. 

Since all the inequalities must be equalities, it follows that $E \cdot \Delta = -\Delta^2$ and $g(B)=1$. 
Furthermore, we have 
$2\pi^*C_0 \cdot \Delta = (E + \Delta) \cdot \Delta$, and thus
\[
0 = \pi^*C_0 \cdot \Delta = \pi^*C_0 \cdot (E + \Delta) = p.
\]
It implies that $E + \Delta= 0$. 
Since $\Delta$ and $E$ are both effective divisors, 
we get $\Delta = 0$ and $E = 0$. 
Hence, $X$ is minimal. 
\end{proof}

By this claim, we can assume that $X$ is a minimal ruled surface over an elliptic curve. 
In this case, it is well-known that $\operatorname{NEF}(X) \subset \operatorname{NE}(X)$ holds
(see Proposition \ref{proposition:ruled}). 
We can conclude that the nef divisor $D$ is $\mbQ$-effective and 
$L$ is semiample by Lemma \ref{lemma:d_effective}. 
\end{proof}

For completeness, we prove two propositions which were used in the above proof. 

\begin{proposition}\label{proposition:abelian}
Let $A$ be an abelian variety defined over an algebraically closed field. 
Then, any nef line bundle on $A$ is numerically equivalent to a semiample line bundle. 
\end{proposition}
\begin{remark}
Note that any effective divisor on an abelian variety is always semiample 
(see the proof of Application 1 ((i)$\Rightarrow$(iii)) in \cite[Section 6]{Mumford:AV}). 
\end{remark}

\begin{proof}
Let $L$ be a nef line bundle on $A$. 
Define $K(L)$ to be the maximal subscheme of $A$ such that 
\[
(m^* L - p_1^* L  - p_2 ^* L )|_{K(L) \times A} = \mcO _{K(L) \times A} 
\]
as in \cite[Section 13]{Mumford:AV}, 
where $m \colon A \times A \to A$ is the multiplication map, 
and $p_i$ are the first and second projections. 

By the above remark, we may assume that $L$ is not big, so that $L^g = 0$ where $g = \dim A$. 
By the Riemann-Roch theorem \cite[Section 16]{Mumford:AV}, we have $\chi (L) = 0$. 
Hence, it follows that $\dim K(L) > 0$ by the vanishing theorem \cite[Section 16]{Mumford:AV}. 

Set $X \coloneq K(L)^0 _{\text{red}}$. This is a subabelian variety of $A$. 
Thus, there exists a subabelian variety $Y \subset A$ such that the morphism 
$m \colon X \times Y \to A; \ (x, y) \mapsto x+y$ defined by the group law
on $A$
is an isogeny (see Theorem 1 in \cite[Section 19]{Mumford:AV}). 
Note that, $L|_X \in \Pic^0(X)$, because it is invariant under translations by any element of $X$ 
(see Remark \ref{remark:pic0abelian}).

First, we prove $m^* L \equiv p_Y ^* (L|_{Y})$, where $p_Y \colon X \times Y \to Y$ is the second projection. 
By definition of $K(L)$, we get $m^*L = p_X^*(L|_X) + p_Y ^*(L|_{Y})$. 
Since $L|_X \in \Pic^0(X)$, we have $L|_{X} \equiv 0$, which proves $m^* L \equiv p_Y ^* (L|_{Y})$. 

Since $\dim Y < \dim A$, we may assume that $L|_Y$ is numerically equivalent to a semiample line bundle 
by induction on $\dim A$. 
By Proposition \ref{proposition:cover}, in order to complete the proof, 
it is sufficient to show that $p_Y ^* (L|_{Y})$ descends to $A$. 
This is true, because $\Pic^0(A) \rightarrow \Pic^0(X \times Y)$ is surjective 
(cf.~ Theorem 1 in \cite[Section 15]{Mumford:AV}).
\end{proof}
\begin{remark} \label{remark:pic0abelian}
Mumford in \cite[Section 8]{Mumford:AV} defines $\Pic^0(X)$, 
for an abelian variety $X$, to be the subgroup of $\Pic(X)$ consisting 
of line bundles invariant under translations by any element of $X$. 
The existence of the dual abelian variety and the Poincar\'e line bundle 
(cf.\ \cite[Section 13]{Mumford:AV}) shows that this definition 
is equivalent to the standard definition of $\Pic^0(X)$ as the identity component of the Picard functor.    
\end{remark}
\begin{proposition}\label{proposition:ruled}
Let $X$ be a minimal ruled surface over an elliptic curve $B$. 
Then, it follows that $\operatorname{NEF}(X) \subset \operatorname{NE}(X)$. 
\end{proposition}
\begin{proof}
We refer the reader to \cite[Ch V, Section 2]{Hartshorne} 
for the properties of ruled surfaces.
Let $C_0 \subset X$ be a normalized section and $F$ a  fiber of $X \to B$. 
Set $e \coloneq - C_0 ^2$. 
When $e \ge 0$, we get
\[
\operatorname{NEF}(X) = \Cone (F, C_0 + e F), 
\]
and so nef line bundles are effective. 

In what follows, we may assume $e = -1$ 
by \cite[Ch V, Theorem 2.15]{Hartshorne}. 
We know that
\[
\operatorname{NEF}(X) = \NE(X)= \Cone (F, 2C_0 - F)
\]
by \cite[Ch V, Proposition 2.21]{Hartshorne}. 
Further, there exists a rank two indecomposable vector bundle $E$ of degree one on $C$ such that
$X \cong \mbP _C (E)$ holds. We denote by $p \colon \mbP _C (E) \to C$ the projection. 
It is sufficient to show that $H^0 (X, \mcO _X (2C_0 - p^* Q)) \not = 0$ for some point $Q \in C$, 
because then $\NE(X) = \operatorname{NE}(X)$. 
Note that 
\[
H^0 (X, \mcO _X (2C_0 - p^* Q)) \cong H^0 (C, S^2(E) \otimes \mcO _C (- Q))
\]
and $S^2(E)$ has both rank and degree equal to three 
(cf.\ \cite[Ch II, Ex 5.16]{Hartshorne} and 
the proof of \cite[Ch V, Theorem 2.15]{Hartshorne}). 
When $S^2(E)$ is indecomposable, we can complete the proof by using the following proposition from Atiyah. 

\begin{proposition}[Atiyah, {\cite[Lemma 11]{Atiyah}}]
Let $F$ be an indecomposable vector bundle of rank $r$ and degree $d$ on an elliptic curve. 
If $r = d$, then $F$ contains a degree one line bundle as a subbundle. 
\end{proposition}

\noindent
When $S^2(E)$ is decomposable, 
it can be written as $S^2(E) \cong E_1 \oplus E_2$, 
where $E_1$ is a line bundle and $E_2$ is a vector bundle of rank two. 
If $\deg E_1 \ge 1$, then 
\[
H^0 (C, S^2(E) \otimes \mcO _C (- Q)) \supset H^0 (C, E_1 \otimes \mcO _C (- Q)) \not = 0
\]
for some point $Q \in C$, which finishes the proof in this case. 
If $\deg E_1 < 1 $, then $\deg E_2 \ge 3$, and so $\deg (E_2 \otimes \mcO _C (- Q)) \ge 1$ for any point $Q \in C$. 
Therefore,
\[
H^0 (C, S^2(E) \otimes \mcO _C (- Q)) \supset H^0 (C, E_2 \otimes \mcO _C (- Q)) \not = 0
\]
by the Riemann-Roch theorem. 
\end{proof}

\section{Reduction to surfaces}\label{section:reduction}

The first step in the proof of Theorem \ref{theorem:main} is to reduce the problem to the case of surfaces.
\begin{theorem} \label{theorem:reduction}
Let $(X, \Delta)$ be a three-dimensional projective log pair defined over $\Fp$, and $L$ a line bundle on $X$. 
If we assume that
\begin{itemize}
\item $L$ and $L - (K_X + \Delta)$ are nef and big, 
\item $L|_{ \Supp \lfloor \Delta \rfloor}$ is semiample, 
\end{itemize}
then $L$ is semiample. 
\end{theorem}
Here, we adopt the convention that, when $\lfloor \Delta \rfloor =0$, then $L|_{ \Supp \lfloor \Delta \rfloor}$ is automatically semiample. 
\begin{remark}
Under the assumption $\lfloor \Delta \rfloor =0$, Theorem \ref{theorem:reduction}
was proved by Keel \cite[Theorem 0.5]{Keel}. 
\end{remark}
\begin{proof}[Proof of Theorem \ref{theorem:main}]
Set $S \coloneq \lfloor \Delta \rfloor$. 
Since $L$ is a big line bundle, we can decompose it as $L \sim_{\mathbb{Q}} A + E$, 
where $A$ is an ample and $E$ is an effective $\mbQ$-Cartier $\mbQ$-divisor. By Theorem \ref{theorem:basereduction} 
it is enough to show that $L|_{E_{\text{red}}}$ is semiample.

We write $E_{\text{red}} = T + \sum_{i=1}^m E_i$, 
where $\Supp (T) \subset \Supp(S)$ and $E_i$ are prime divisors not contained in $\Supp(S)$.
Define $\lambda_i \in \mathbb{Q}$ so that $\Delta + \lambda_i E$ contains $E_i$ with coefficient one. 
Then by definition of $\lambda_i$, there exists 
an effective $\mathbb{Q}$-divisor $\Gamma_i$ such that 
\[
\Delta + \lambda_i E = E_i + \Gamma_i
\]
and $E_i \not \subset \Supp(\Gamma_i)$. 
Since $E_i \not \subset \Supp(S)$, 
it follows that $\lambda_i > 0$.
By rearranging indices, we may assume without loss of generality that
\[
\lambda_1 \leq \lambda_2 \leq \ldots \leq \lambda_m,
\]
so we have 
\[
T + \sum_{1 \le j \le i-1} E_j \leq \Gamma_i 
\]
for each $i$. 

We define $U_0 \coloneq \Supp(T)$ and $U_i \coloneq U_{i-1} \cup E_i$ for $i>0$. 
Recall that it is sufficient to show that $L$ restricted to $U_m = \Supp(E_{\text{red}})$ is semiample. 
We prove it by induction on $i$.

Observe that $L|_{U_0}$ is semiample, 
because $U_0 = \Supp(T) \subset \Supp(S)$ and $L|_{S}$ is semiample by hypothesis. 
Let us assume that $L|_{U_{i-1}}$ is semiample. 
In order to prove the semiampleness of $L|_{U_i}$, we first prove the semiampleness of $L|_{E_i}$.

We consider the normalization $p_i\colon \overline{E_i} \to E_i$.
By adjunction (see Subsection \ref{subsection:adjunction}), 
there exists an effective $\mbQ$-divisor $\Delta _{\overline{E_i}}$ such that 
\[
(K_X + E_i + \Gamma _i)|_{\overline{E_i}} \sim K_{\overline{E_i}} + \Delta _{\overline{E_i}}
\]
holds. 
\begin{lemma}\label{lemma:semiample_normal}
$L|_{\overline{E_i}}$ is semiample.
\end{lemma}

\begin{proof}
We define auxiliary divisors $D_i$ as follows:
\[
	D_i \coloneq (1 + \lambda_i)L - (K_X + \Delta + \lambda_i E).  
\]
Observe that
\begin{align*}
	D_i & =L - (K_X + \Delta) + \lambda_i (L - E) \\
		&\sim_{\mathbb{Q}} (L - (K_X + \Delta)) + \lambda_i A,
\end{align*}
and so $D_i$ is ample, because $L - (K_X + \Delta)$ is nef and $\lambda_i A$ is ample. 
Hence,
\[
D_i |_{\overline{E_i}} = (1 + \lambda_i)L| _{\overline{E_i}} - (K_{\overline{E_i}} + \Delta_{\overline{E_i}})
\]
is nef. Since $(1+\lambda_i)L|_{\overline{E_i}}$ is also nef, 
the semiampleness of $L|_{\overline{E_i}}$ follows from Theorem \ref{theorem:genshokurov} and Remark \ref{remark:qcartier}.
\end{proof}

Assume $\kappa(L|_{\overline{E_i}})$ is equal to $0$ or $2$. 
Then the assumptions of Theorem \ref{theorem:gluing2} are satisfied, and so $L|_{E_i}$ is semiample. 
Using Theorem \ref{theorem:gluing1} for $X_1 = U_{i-1}$ and $X_2 = E_i$, we get that $L|_{U_i}$ is semiample. 

In what follows, we assume $\kappa(L|_{\overline{E_i}}) = 1$.
\begin{lemma}\label{lemma:FDelta<2}
Let $\pi _i \colon \overline{E_i} \to Z_i$ be 
the map associated to the semiample line bundle $L|_{\overline{E_i}}$ and 
let $F$ be a general fiber of $\pi _i$. 
Further, let $C_i \subset \overline{E_i}$ be the the reduction of the conductor of the normalization 
$p_i \colon \overline{E_i} \to E_i$.
Then $F$ and $C_i$ intersect in at most one point.  
\end{lemma}
\begin{proof}
Let $D_i$ be the $\mbQ$-divisor on $\overline{E_i}$ as in the proof of Lemma \ref{lemma:semiample_normal}. 
Then, $D_i$ is ample, so we have $F \cdot D_i|_{\overline{E_i}} > 0$. 
Since $F \cdot L|_{\overline{E_i}}=0$, 
we get
\[
F \cdot K_{\overline{E_i}} + F \cdot \Delta_{\overline{E_i}} < 0.
\] 
Hence, 
\[
F \cdot \Delta_{\overline{E_i}} < -F \cdot K_{\overline{E_i}} 
= 2 - 2 h^1 (F, \mcO_F) \leq 2
\]
holds.

By the adjunction formula (Proposition \ref{proposition:adjunction}), 
the one-dimensional part of $C_i$ is contained in $\Supp (\lfloor \Delta_{\overline{E_i}} \rfloor)$. 
Hence, we get $\# (F \cap C_i) \leq F \cdot \Delta_{\overline{E_i}} < 2$. 
\end{proof}

By this lemma, the assumptions of Theorem \ref{theorem:gluing2} are satisfied, and so $L|_{E_i}$ is semiample. 
Let $\rho _i \colon E_i \to Z_i'$ be the map associated to $L|_{E_i}$ 
and let $G$ be a general fiber of $\rho _i$. 
Since $\pi _i$ is the Stein factorization of $\rho _i \circ p_i$, 
there exists a finite map $Z_i \to Z_i '$ such that the following diagram commutes 
(cf.\ \cite[Definition-Lemma 1.0 (4)]{Keel}). 
\[\xymatrix{
\overline{E_i} \ar[d] _{\pi _i} \ar[r] ^{p_i} & E_i \ar[d] ^{\rho _i}  \\
Z_i \ar[r]  & Z_i '
}\]

We want to apply Theorem \ref{theorem:gluing1} to $X_1 = U_{i-1}$ and 
$X_2 = E_i$ to show that $L|_{U_i}$ is semiample. 
It is sufficient to prove that $G$ intersects $U_{i-1}\cap E_i$ in at most one point. 

Recall that
\[
T + \sum _{1 \le j \le i-1} E_j \le \Gamma _i, \ \ \; 
U_{i-1} = \Supp \Big( T + \sum _{1 \le j \le i-1} E_j \Big).
\]
Hence, the one-dimensional part of 
$p_i^{-1}(U_{i-1} \cap E_i)$ is contained in $\Supp(\lfloor \Delta_{\overline{E_i}} \rfloor)$ 
by the adjunction formula (Proposition \ref{proposition:adjunction}). 
By the proof of Lemma \ref{lemma:FDelta<2}, we can conclude 
\begin{align*}
\# \big((U_{i-1} \cap E_i) \cap G \big) 
&= \# \bigg( p_i \big(p_i ^{-1}(U_{i-1} \cap E_i) \cap F \big) \bigg) \\
&\leq \# \big(p_i ^{-1}(U_{i-1} \cap E_i) \cap F \big) \\
&\leq F \cdot \Delta_{\overline{E_i}} < 2, 
\end{align*}
which completes the proof. 
\end{proof}

\section{Semiampleness on non irreducible surfaces}\label{section:glue}

In this section, we prove Theorem \ref{theorem:2dim}. 
Before stating it, we need to introduce some notation. 
Let $S$ be a pure two-dimensional reduced projective scheme over $\Fp$ and 
let $S=\bigcup_{i=1}^n S_i$ be its irreducible decomposition, and $\overline{S} \to S$ its normalization. 
Let $\mcD \subset S$ and $\mcC \subset \overline{S}$ be the conductors of $S$. 
Let $\overline{C} \xrightarrow{\text{norm.}} \mcC _{\text{red}} \longrightarrow \mcC$ and 
$\overline{D} \xrightarrow{\text{norm.}} \mcD _{\text{red}} \longrightarrow \mcD$ 
be the compositions of the reduction map and the normalization. 
Then we have a natural morphism $f \colon \overline{C} \to \overline{D}$ 
such that the following diagram commutes. 
\[\xymatrix{
\overline{C} \ar[d] _{f} \ar[r] ^{\text{norm.}\ } & \mcC _{\text{red}} \ar[r] & \mcC \ar@<-0.5mm>@{^(->}[r] \ar[d] & \overline{S} \ar[d] \\
\overline{D} \ar[r] ^{\text{norm.}\ } & \mcD _{\text{red}} \ar[r] & \mcD \ar@<-0.5mm>@{^(->}[r] & S
}\]

Consider the one-dimensional part $\overline{C}^{(1)}$ of $\overline{C}$ and the restriction 
$f \colon \overline{C}^{(1)} \to \overline{D}$. 
We say that \textit{$S$ satisfies the condition} ($\star$) 
when the restriction of $f$ to any one-dimensional connected component of $\overline{C}$ 
is an isomorphism onto its image. 
Further, we say that $S$ \textit{satisfies the condition} ($\star \star$) when 
any fiber of the restriction $f \colon \overline{C}^{(1)} \to \overline{D}$ 
has length at most two.

\begin{remark}\label{remark:star}
If each $S_i$ is normal, then $S$ satisfies the condition $(\star)$. 
If $S$ is regular or nodal in codimension one, 
then $S$ satisfies the condition $(\star \star)$. 
See the proof of Theorem \ref{theorem:main}. 
\end{remark}

\begin{theorem}\label{theorem:2dim}
Let $S$ be a pure two-dimensional reduced projective scheme over $\Fp$ and 
let $S=\bigcup_{i=1}^n S_i$ be its irreducible decomposition. 
Let $L$ be a nef line bundle on $S$. 
Suppose that $S$ satisfies 
the condition $(\star)$ or $(\star \star)$
defined above and that there exists an effective 
$\mbQ$-divisor $\Delta_{\overline{S}}$ 
on the normalization $\overline{S}$ of $S$ such that 
\begin{itemize}
\item $L|_{\overline{S}} -(K_{\overline{S}} + \Delta_{\overline{S}})$ is nef, and
\item $\Supp(\mcC^{(1)})$ is contained in $\Supp (\lfloor \Delta _{\overline{S}} \rfloor)$, 
where $\mcC^{(1)} \subset \overline{S}$ is the one-dimensional part of the conductor scheme of the normalization of $S$. 
\end{itemize}
Then $L$ is semiample. 
\end{theorem}

\begin{proof}
We use the same notation as above. 
Let $\nu \colon \overline{S} \coloneq \bigsqcup \overline{S_i} \to S$ be the normalization of $S$. 
Set $\Delta_{\overline{S_i}} \coloneq \Delta_{\overline{S}}|_{\overline{S_i}}$. 
We know that $L|_{\overline{S_i}}$ are semiample from Theorem \ref{theorem:genshokurov}. 
Let $g_i \colon \overline{S_i} \to Z_i$ 
be the map associated to $L|_{\overline{S_i}}$. 
Set $g \colon \overline{S} \to Z $, 
where $g \coloneq \sqcup g_i$ and $Z \coloneq \bigsqcup Z_i$. 
If $\dim Z_i \not = 1$, then $g_i$ satisfies the conditions of Theorem \ref{theorem:gluing1}. 
Hence, we may assume that $\dim Z_i = 1$ for any $i$ 
by the inductive argument as in the proof of Theorem \ref{theorem:reduction}. 
\[\xymatrix{
\overline{C} \ar@/^6mm/[rrrr] ^{f_2} \ar[d] _{f_1} \ar[r] _{\text{norm.}\ } & \mcC _{\text{red}} \ar[r] & \mcC \ar@<-0.5mm>@{^(->}[r] \ar[d] & \overline{S} \ar[r] _g \ar[d] ^{\nu} & Z \\
\overline{D} \ar[r] ^{\text{norm.}\ } & \mcD _{\text{red}} \ar[r] & \mcD \ar@<-0.5mm>@{^(->}[r] & S &
}\]

By Remark \ref{remark:conductor}, it is sufficient to show that for any point $p \in \overline{S}$, 
there exist $m \ge 1$ and a section 
$s \in H^0 (\overline{S}, L^{\otimes m} | _{\overline{S}})$
such that $s|_{\mcC}$ descends to $\mcD$ and $s|_p \not = 0$. 
To obtain this, we prove the following claim. 

\begin{claim}\label{claim:conductor_normal}
For any finite set $F \subset \overline{S}$ of closed points of $\overline{S}$, 
we can find $m \ge 1$ and a section $s \in H^0 (\overline{S}, L^{\otimes m} | _{\overline{S}})$ such that 
$s|_{\overline{C}}$ descends to $\overline{D}$ and $s$ is nowhere vanishing on $F$. 
\end{claim}

First, we assume this claim and complete the proof of Theorem \ref{theorem:2dim}. 
Let $F' \subset \mcD _{\text{red}}$ be the conductor corresponding 
to the normalization $\overline{D} \to \mcD _{\text{red}}$. 
Let $F''$ be the image of $F'$ in $S$. Set $F \coloneq \nu^{-1}(F'') \cup \{ p \}$. 
Then $F$ is a finite set. 

By Claim \ref{claim:conductor_normal}, we can take
$s \in H^0 (\overline{S}, L^{\otimes m} | _{\overline{S}})$ and 
$s_{\overline{D}} \in H^0 (\overline{D}, L^{\otimes m} | _{\overline{D}})$ such that 
$s |_{\overline{C}} = s_{\overline{D}} |_{\overline{C}}$ and that $s$ is nowhere vanishing on $F$. 
By Lemma \ref{lemma:descent}, if we replace $s_{\overline{D}}$ by some power of it, then 
$s_{\overline{D}}$ descends to a section $s_{\mcD _{\text{red}}}$ on $\mcD _{\text{red}}$. 
Since $\mcD _{\text{red}} \to \mcD$ is a universal homeomorphism, 
$s_{\mcD _{\text{red}}}$ descends to a section $s_{\mcD}$ on $\mcD$, 
if we replace $s_{\overline{D}}$ by some power of it (cf.~ Theorem \ref{theorem:univhomeo}). 

It is sufficient to show that $s|_{\mcC} = s_{\mcD}|_{\mcC}$. 
By construction, $(s|_{\mcC})|_{\overline{C}} = (s_{\mcD}|_{\mcC})|_{\overline{C}}$ holds. 
Since $\overline{C} \to \mcC _{\text{red}}$ is surjective, 
we get $(s|_{\mcC})|_{\mcC _{\text{red}}} = (s_{\mcD}|_{\mcC})|_{\mcC _{\text{red}}}$. 
As $\mcC _{\text{red}} \to \mcC$ is a universal homeomorphism, if we replace $s$ by some power of it, 
then we get $s|_{\mcC} = s_{\mcD}|_{\mcC}$ (cf.~ Theorem \ref{theorem:univhomeo}). 
This completes the proof of Theorem \ref{theorem:2dim}. 

\begin{proof}[Proof of Claim \ref{claim:conductor_normal}]
Let $f_1$ and $f_2$ be as in the above diagram. 
For a one-dimensional scheme $X$, we write $X = X^{(0)} \sqcup X^{(1)}$, 
where $X^{(i)}$ is the $i$-dimensional part. 
Further, we write $\overline{C} ^{(1)} = \overline{C} ^{\text{h}} \sqcup \overline{C} ^{\text{v}}$, 
where $\overline{C} ^{\text{h}}$ is the $f_2$-horizontal part and 
$\overline{C} ^{\text{v}}$ is the $f_2$-vertical part. 

First, we claim that for any closed point $p \in Z$ 
the inverse image of $p$ by $\overline{C} ^{\text{h}} \to Z$ has length at most two. 
This can be proved as follows: 
by the nefness of $L-(K_{\overline{S_i}} + \Delta _{\overline{S_i}})$, we have
\[
0 \le G_i \cdot (L-(K_{\overline{S_i}} + \Delta _{\overline{S_i}}))
= - G_i \cdot (K_{\overline{S_i}} + \Delta _{\overline{S_i}}) \le 2 - G_i \cdot \Delta _{\overline{S_i}}, 
\]
where $G_i$ is a general fiber of $g_i \colon \overline{S_i} \to Z_i$. 
Since the one-dimensional part of $\mcC |_{\overline{S_i}}$ is contained in 
$\Supp (\lfloor \Delta _{\overline{S_i}} \rfloor)$, we have 
\[
\# (G_i \cap \mcC |_{\overline{S_i}}) \le G_i \cdot \Delta _{\overline{S_i}} \le 2. 
\]

Hence, $f_2 \colon \overline{C} ^{\text{h}} \to Z$ satisfies
the assumption of Lemma \ref{lemma:double}. 
Further, by the condition 
$(\star)$ and $(\star \star)$, 
$f_1 \colon \overline{C} ^{\text{h}} \to D'$ also satisfies the assumption of Lemma \ref{lemma:double}, 
where we define $D' \coloneq f_1 (\overline{C} ^{\text{h}})$. 
\[\xymatrix{
\overline{C} = \overline{C} ^{\text{h}} \sqcup \overline{C} ^{\text{v}} \sqcup \overline{C} ^{(0)} \ar[d] _{f_1} \ar[r] ^-{f_2} & Z & 
\hspace{0mm} &
\overline{C} ^{\text{h}} \ar[d] _{f_1} \ar[r] ^{f_2} & Z\\
\overline{D}=D' \sqcup (\overline{D} \setminus D') & \hspace{0mm} & & D' &
}\]

By Lemma \ref{lemma:double}, 
we can find sections 
$s_{\overline{D}} \in H^0(\overline{D}, L^{\otimes m} |_{\overline{D}})$ and 
$s_{Z} \in H^0(Z, L^{\otimes m} |_{Z})$ such that 
$s_{\overline{D}}|_{\overline{C} ^{\text{h}}}=s_{Z}|_{\overline{C} ^{\text{h}}}$ holds, $s_{Z}$ is nowhere vanishing on 
the finite set $g(F) \cup f_2(\overline{C} ^{\text{v}} \sqcup \overline{C} ^{(0)})$, and 
$s_{\overline{D}}$ is nowhere vanishing on $\overline{D} \setminus D'$. 
Since $L|_{\overline{C} ^{\text{v}} \sqcup \overline{C} ^{(0)}}$ is trivial, we have
$s_{\overline{D}} ^n|_{\overline{C} ^{\text{v}} \sqcup \overline{C} ^{(0)}}=s_{Z} ^n|_{\overline{C} ^{\text{v}} \sqcup \overline{C} ^{(0)}}$ 
for some $n \ge 1$ by Lemma $\ref{lemma:0dim}$.
Therefore, we get $s_{\overline{D}} ^n |_{\overline{C}} = s_{Z} ^n|_{\overline{C}}$ 
and this completes the proof of Claim \ref{claim:conductor_normal}. 
\end{proof}
\vspace{-2mm}
\end{proof}

Finally, we show the following lemma, which was used in the proof of Theorem \ref{theorem:2dim}. 
\begin{lemma}\label{lemma:double}
Let $X, Z_1, Z_2$ be disjoint unions of smooth proper curves, and
$f_1 \colon X \to Z_1$, $f_2 \colon X \to Z_2$ be finite surjective morphisms. 
Let $L_1$ and $L_2$ be line bundles on $Z_1$ and $Z_2$, respectively, 
such that $f_1 ^* L_1 = f_2 ^* L_2$. 
Suppose that $L \coloneq f_1 ^* L_1 = f_2 ^* L_2$ is semiample.  
Further, assume that each $f_i$ satisfies either of the following conditions:
\begin{itemize}
\item the restriction of $f_i$ to any connected component of $X$ is an isomorphism onto its image, or 
\item any fiber of $f_i$ has length at most two. 
\end{itemize}

Then, for any finite set $F \subset X$ of closed points of $X$, 
we can take $m \ge 1$ and a section $s \in H^0 (X, L^{\otimes m})$ 
such that $s$ is nowhere vanishing on $F$ and $s$ descends to both $Z_1$ and $Z_2$. 
\end{lemma}

\begin{proof}
First, we prove that there exists a finite group $G_i$ acting on $X$ such that 
$X \to Z_i$ decomposes into the quotient morphism $X \to X/G_i$ and a universal homeomorphism $X/G_i \to Z_i$. 
\[\xymatrix{
&&& \ar[ld] \ar[llld]_{f_1} X \ar[rd] \ar[rrrd]^{f_2} &&&\\
Z_1 && \ar[ll]^{\hspace{-2mm} \text{univ.\ homeo}} X/G_1 && X/G_2 \ar[rr]_{\hspace{2mm} \text{univ.\ homeo}} && Z_2
}\]

This is trivial when the restriction of $f_i$ to any connected component of $X$ is an isomorphism. 
Indeed, it is sufficient to take $G_i$ such that 
it identifies the components with the same image under $f_i$. 
Then $X \to Z_i$ is isomorphic to the quotient morphism $X \to X/G_i$. 

For the second case, assume that any fiber of $f_i$ has length at most two. 
Let $Z_i '$ be a connected component of $Z_i$. Set $X' = f_i ^{-1} (Z_i ')$. 
There are four possibilities:
\begin{enumerate}
\item[(1)] $X'$ is connected and $X' \to Z_i '$ is an isomorphism. 
\item[(2)] $X'$ is connected and $X' \to Z_i '$ is the Frobenius map 
(this case may only occur for characteristic $p=2$). 
\item[(3)] $X'$ is connected and every fiber of $X' \to Z_i '$ has length two. 
There exists an involution $\iota \colon X' \to X'$ such that $X' \to Z_i '$ 
is the quotient by $\iota$. 
\item[(4)] $X'$ has two connected components $X_1 '$ and $X_2 '$. 
Further, $X' _1 \to Z_i '$ and $X' _2 \to Z_i '$ are isomorphisms. 
In this case, we have $X_1 ' \cong X_2 '$. 
\end{enumerate}
In the cases (3) and (4), 
we have a finite group $G'$ acting on $X'$ such that the morphism $X' \to Z_i '$ 
is isomorphic to the quotient morphism $X' \to X' / G'$. 

Hence, we have a finite group $G_i$ acting on $X$ such that the morphism $X \to Z_i$ decomposes in
the following way $X \to X/ G_i \to Z_i$, 
where $X \to X/G_i$ is the quotient morphism and $X/G_i \to Z_i$ 
is a universal homeomorphism 
(actually, if we restrict it to a connected component, it is either an isomorphism or the Frobenius map). 

Note that $L = g^* L$ for any $g \in G_i$. 
We claim that if $s \in H^0(X, L^{\otimes m})$ is $G_i$-equivariant, then
$s^ {p^e}$ descends to $Z_i$ for sufficiently large $e$. 
This is because $s$ descends to $X/G_i$ and $X/G_i \to Z_i$ is a universal homeomorphism 
(cf.~ Theorem \ref{theorem:univhomeo}). 

Let $G \coloneq G_1 G_2 \subset \Aut (X)$ be a composition of the groups and 
let $S \subset X$ be the $G$-orbit of the set $F$. 
By Lemma \ref{lemma:fin_group}, $G$ is a finite group, and therefore $S$ is a finite set. 

Take $m \ge 1$ and a section $s \in H^0 (X, L^{\otimes m})$ such that $s$ is nowhere vanishing on $S$. 
Set
\[
s^G \coloneq \prod _{\sigma \in G} \sigma ^* s \in H^0 (X, L^{\otimes m|G|}). 
\]
The section $s^G$ is $G_i$-invariant for each $i$ and nowhere vanishing on $F$. 
Hence, $(s^G)^{p^e}$ satisfies the statement of the lemma for sufficiently large $e \ge 1$. 
\end{proof}

\begin{remark}
The main issue of this section is related to the following question discussed by Keel in \cite{Keel2003}. 
\begin{question}
Let $L$ be a line bundle on a variety $X$ and let $p \colon \overline{X} \to X$ be the normalization of $X$. Assume that $p^*L$ is semiample. 
What additional assumptions are necessary for $L$ to be semiample?
\end{question}
\end{remark}

\section{Proof of Theorem \ref{theorem:main}}\label{section:proof}

\indent
In this section, we prove Theorem \ref{theorem:main} using Theorem \ref{theorem:reduction} and 
Theorem \ref{theorem:2dim}. 

\begin{proof}[Proof of Theorem \ref{theorem:main}]
Let $S \coloneq \lfloor \Delta \rfloor$. 
By Theorem \ref{theorem:reduction}, it is sufficient to show that $L |_{\Supp(S)}$ is semiample. 
Note that in both case (1) and case (2), all the coefficients of $\Delta$ are at most one. 

By the adjunction formula (Proposition \ref{proposition:adjunction}), 
if we define $\Delta_{\overline{S}}$ on $\overline{S}$ so that 
$(K_X + \Delta)|_{\overline{S}} = K_{\overline{S}} + \Delta_{\overline{S}}$, then 
$\Delta_{\overline{S}}$ satisfies the conditions in the statement of Theorem \ref{theorem:2dim}. 

In the case (2), that is, the case when each component $S_i$ of $S$ is normal, 
$S$ clearly satisfies the condition $(\star)$. 
In the case (1), that is, the case when $(X, \Delta)$ is log canonical, 
the surface $S$ is regular or nodal in codimension one (see \cite[Corollary 2.32]{Kollar}), 
and so 
$S$ satisfies the condition $(\star \star)$ 
(see \cite[Claim 1.41.1]{Kollar} or \cite[Lemma 3.4, 3.5]{TanakaAbundanceForSlc}).

Thus, we can complete the proof by using Theorem \ref{theorem:2dim}. 
\end{proof}

We easily deduce Corollary \ref{corollary:abundance}. 
\begin{proof}[Proof of Corollary \ref{corollary:abundance}]
It is enough to take $L = 2(K_X + \Delta)$ and $L = -(K_X + \Delta)$, respectively.
\end{proof}

\section{Examples}\label{section:example}
Theorem \ref{theorem:main} does not hold if we do not impose any conditions on $\Delta$. 
It is in fact possible to construct a nef and big line bundle $L$ on a smooth threefold $X$ 
such that $L-(K_X+\Delta)$ is nef and big for $\Delta \geq 0$, but $L$ is not semiample. 
We construct such $L$ and $\Delta$ in the following way.
\begin{example} \label{example:counterex}
Let $L$ be a nef and big line bundle on a smooth threefold, 
which is not semiample (see an example of Totaro in \cite[Theorem 7.1]{Totaro}). 
Since $L$ is big, we can write $L = A + E$ 
for an ample $\mbQ$-Cartier $\mbQ$-divisor $A$ 
and an effective $\mbQ$-Cartier $\mbQ$-divisor $E$. 
Take $\Delta = mE$ for $m \in \mbN$ big enough. 
Then $mL - (K_X+\Delta)$ is an ample Cartier divisor, 
and so the pair $L' \coloneq mL$ and $\Delta$ is an example, 
which we were looking for.  
\end{example}

Theorem \ref{theorem:main} does not hold 
over algebraically closed fields $k \not = \Fp$ even in the two-dimensional case. 
\begin{example}[Tanaka {\cite[Example 19.3]{Tanaka}}]\label{example:tanaka}
Let $C_0 \subset \mbP ^2$ be an elliptic curve in $\mbP ^2$, and 
let $p_1, \ldots, p_{10} \in C_0$ be ten general points on $C_0$. 
Let $X$ be the blowup of $\mbP ^2$ along these ten points, and $C$ the proper transform of $C_0$. 
Note that $K_X + C \sim 0$ and $C^2 = -1$. 

Take an ample divisor $H$ on $X$, and
set $L \coloneq H + a C$, where $a \coloneq H \cdot C > 0$. 
Note that $L$ is a nef and big divisor. 
Further, $(X, C)$ is log canonical, and $L - (K_X + C)$ is also nef and big. 
Nevertheless, $L$ is not semiample if the base field is not $\Fp$. This is because $L \cdot C = 0$, 
but the elliptic curve $C$ is not contractible. 
\end{example}

Corollary \ref{corollary:abundance} (2) also does not hold 
over algebraically closed fields $k \not = \Fp$. 
\begin{example}[Gongyo {\cite[Example 5.2]{Gongyo}}]\label{example:gongyo}
Let $S$ be the blowup of $\mbP ^2$ along nine general points. 
Note that $- K_S$ is nef but not semiample if the base field is not $\Fp$. 
Take a very ample divisor $H$ on $S$, 
and set $X \coloneq \mbP _S (\mcO _S \oplus \mcO _S (-H))$. 
Let $E$ be the tautological section of $\mcO _S \oplus \mcO _S (-H)$. 
Since $E \cong S$, it follows that $-K_E$ is not semiample. 

Then, $(X, E)$ is log canonical, and  
$L \coloneq - (K_X + E)$ is nef and big by the nefness of $- K_S$ (for details, see \cite[Example 5.2]{Gongyo}). 
Nevertheless, $L$ is not semiample, because $L |_E  = - K_E$ is not semiample. 
\end{example}

\end{document}